\documentclass[11pt]{article}
\newcommand\version{October 10, 2026}

\usepackage[T1]{fontenc}
\usepackage{lmodern}
\usepackage{amsmath,amssymb,amsthm,mathtools}
\usepackage{microtype}
\usepackage[letterpaper,margin=1in]{geometry}
\usepackage{hyperref}
\hypersetup{colorlinks=true,linkcolor=black,citecolor=black,urlcolor=black,
  pdftitle={The sharp constant in the Chang--Marshall inequality},
  pdfauthor={Rupert L. Frank and Paata Ivanisvili}}
\allowdisplaybreaks
\newtheorem{theorem}{Theorem}
\newtheorem{proposition}[theorem]{Proposition}
\newtheorem{lemma}[theorem]{Lemma}
\newtheorem{corollary}[theorem]{Corollary}
\theoremstyle{remark}

\newcommand{\D}{\mathbb{D}}
\newcommand{\T}{\mathbb{T}}
\newcommand{\C}{\mathbb{C}}
\newcommand{\dd}{\,\mathrm{d}}
\DeclareMathOperator{\Rea}{Re}
\DeclareMathOperator{\Ima}{Im}

\title{The sharp constant in the\\Chang--Marshall inequality}

\author{Rupert L. Frank\thanks{Mathematisches Institut, Ludwig-Maximilians Universit\"at M\"unchen, Theresienstr.~39, 80333 M\"unchen, Germany; Munich Center for Quantum Science and Technology, Schellingstr.~4, 80799 M\"unchen, Germany. Email: \texttt{r.frank@lmu.de}.}
\and
Paata Ivanisvili\thanks{Department of Mathematics, University of California, Irvine, 510C Rowland Hall, Irvine, CA 92697-3875, USA. Email: \texttt{pivanisv@uci.edu}.}}

\date{\version}

\begin{document}

\maketitle

\begin{abstract}
We determine the optimal constant in the Chang--Marshall inequality, thereby proving a conjecture by Andreev and Matheson. More precisely, if $f$ is analytic in the unit disk, vanishes to order $m\geq1$ at the origin, and has normalized Dirichlet integral $D\leq1$, then the boundary mean of $\exp(|f|^2)$ is at most $\exp(D/m)$, with equality precisely for monomials.
\end{abstract}

\noindent\emph{2020 Mathematics Subject Classification.} Primary 30A10; Secondary 30H25, 46E35.

\noindent\emph{Key words and phrases.} Chang--Marshall inequality, Dirichlet space, sharp constant, isoperimetric inequality, radial monotonicity.


\section{Introduction and main results}

For an analytic function $f$ in the unit disk $\D:=\{z\in\C:|z|<1\}$, we set
\begin{equation}\label{eq:energy0}
    D(f):= \int_\D |f'(z)|^2 \, \frac{\dd A(z)}{\pi} \,,
\end{equation}
where $\dd A$ denotes ordinary planar area measure. 

If $D(f)<\infty$, then it is well known that $f$ has almost-everywhere radial limits on the boundary $\partial\D$. The Chang--Marshall inequality \cite{CM} implies exponential integrability of these boundary values with an optimal constant in the exponential. Our goal in this paper is to give a remarkably simple proof of this inequality. In addition, this proof gives the optimal constant on the right side of the inequality, which had been an open problem \cite{AM} for some time.

In fact, we will prove the following, stronger inequality.

\begin{theorem}\label{thm:main}
Let $m\geq 1$ and let $f\not\equiv0$ be analytic in $\D$, having a zero of order $m$ at the origin and satisfying $D(f)\leq 1$. Then
\begin{equation}\label{eq:main}
\int_0^{2\pi} \exp\bigl(|f(e^{i\theta})|^2\bigr) \,\frac{\dd\theta}{2\pi} \leq \exp\!\left(\frac{D(f)} m\right).
\end{equation}
with equality if and only if $f(z)=cz^m$ for some $c\in\C$ with $0<|c|\leq 1/\sqrt{m}$.
\end{theorem}

In particular, we obtain the following sharp form of the Chang--Marshall inequality.

\begin{corollary}\label{cor:sharp}
For every analytic function $f$ in $\D$ with $f(0)=0$ and $D(f)\leq1$,
\begin{equation}\label{eq:sharp}
    \int_0^{2\pi}\exp\bigl(|f(e^{i\theta})|^2\bigr) \, \frac{\dd\theta}{2\pi}\leq e.
\end{equation}
Equality holds if and only if $f(z)=e^{i\alpha}z$ for some $\alpha\in\mathbb{R}$.
\end{corollary}

Indeed, for $f\not\equiv0$, Theorem~\ref{thm:main} gives $J(f)\leq e^{D(f)/m}\leq e$. Equality in the latter inequality requires $D(f)=m=1$, and the equality statement follows. The case $f\equiv 0$ is trivial.


\subsection*{The Chang--Marshall inequality and Andreev--Matheson problem}

Beurling \cite{Beu} showed that the supremum 
\begin{equation}
    \label{eq:supremum}
    \sup\left\{ \int_0^{2\pi}\exp\bigl(\alpha |f(e^{i\theta})|^2\bigr) \, \frac{\dd\theta}{2\pi} :\ f \ \text{analytic in $\D$ with $f(0)=0$ and $D(f)\leq1$} \right\}
\end{equation}
is finite if $\alpha<1$ and infinite if $\alpha>1$. The question whether the supremum in \eqref{eq:supremum} remains finite in the critical case $\alpha=1$ was raised by Moser \cite{Mo}. He was able to answer a similar, but different question concerning Trudinger's result about the exponential integrability of Sobolev functions.

Moser's question was finally settled by Chang and Marshall~\cite{CM}, who proved that the supremum is finite for $\alpha=1$. An alternative proof of this fact was given by Marshall~\cite{M} and an extension appears in \cite{Es}. For background on the Chang--Marshall inequality, we refer to \cite[Section 9.2]{ARSW}.

However, the value of the supremum \eqref{eq:supremum} with $\alpha=1$ remained unknown. Andreev and Matheson~\cite{AM} showed that \(f(z)=z\) is a local maximizer and conjectured that it is also a global maximizer. Subsequent work established extremality in restricted classes \cite{CiMa,An} and investigated regularity and existence in related extremal problems \cite{MP}. In contrast to the the famous Carleson--Chang result~\cite{CC} concerning Moser's inequality for Sobolev functions, even the existence of a maximizer for the supremum \eqref{eq:supremum} with $\alpha=1$ has remained open so far.

Our Corollary~\ref{cor:sharp} proves the conjecture of Andreev and Matheson~\cite{AM} and identifies all maximizers. Theorem~\ref{thm:main} gives a stronger inequality that also accounts for the order of vanishing and value of the Dirichlet integral.


\subsection*{Strategy of the proof}

Let us abbreviate
\begin{equation}
    \label{eq:defj}
    J(f):= \int_0^{2\pi} \exp\bigl(|f(e^{i\theta})|^2\bigr) \,\frac{\dd\theta}{2\pi} \,.
\end{equation}
The key idea of our proof of Theorem \ref{thm:main} is to show that the function
\begin{equation}
    \label{eq:monointro}
     r\longmapsto J\!\left(\frac{f(r\,\cdot)}{\sqrt{D(f(r\,\cdot))}}\right)
\end{equation}
is nonincreasing on the interval $(0,1)$. Comparing the limiting value at $r=0$, which is easy to compute, and the limiting value at $r=1$, which is the desired $J(f)$ when $D(f)=1$, we immediately deduce the claimed inequality in Theorem \ref{thm:main} when $D(f)=1$. The inequality for $D(f)<1$ follows easily from the one for $D(f)=1$.

It turns out that the monotonicity of \eqref{eq:monointro} is a consequence of the isoperimetric inequality for an auxiliary closed, smooth curve. The fact that the curve is not necessarily simple does not create any problems, provided we understand the `enclosed area' in the isoperimetric inequality in the sense of the algebraic signed area.

We recall this version of the isoperimetric inequality and use it to prove the key lemma in Section~\ref{sec:geometry}. In Section \ref{sec:radial}, we deduce the monotonicity of \eqref{eq:monointro} from this key lemma and discuss strictness of the monotonicity. Section~\ref{sec:bound} establishes Theorem \ref{thm:main}.


\section{An auxiliary closed curve}\label{sec:geometry}

We first recall the form of the isoperimetric inequality that we shall use. For a closed continuously differentiable planar curve $\gamma$, we let
\begin{align*}
 L_\gamma &:=\int_\gamma \dd s  
 \qquad\text{and}\qquad
 \mathcal A_\gamma
 :=\frac12\int_\gamma \Ima(\overline{z} \tau)\dd s
\end{align*}
denote its length and algebraic signed area, respectively. Here $\dd s$ is the arclength element and $\tau$ is the tangent vector of the curve, identified with a complex number. Note that $\tau$ is defined $\dd s$-almost everywhere. In terms of a continuously differentiable parametrization $\eta:[0,T]\to\C$ of $\gamma$ we have $\dd s = |\eta'(t)|\dd t$ and $\tau = \eta'(t)/|\eta'(t)|$, so
\begin{align*}
 L_\gamma =\int_0^T |\eta'(t)| \dd t  
 \qquad\text{and}\qquad
 \mathcal A_\gamma=\frac12\int_0^T \Ima(\overline{\eta(t)} \eta'(t))\dd t \,.
\end{align*}
The signed area formulation is useful because our curve need not be simple.

\begin{lemma}\label{lem:iso}
Let $\gamma$ be a closed continuously differentiable planar curve. Then
\begin{equation}\label{eq:iso}
 4\pi|\mathcal A_\gamma|\leq L_\gamma^2.
\end{equation}
\end{lemma}

We recall the standard argument that appears, for example, in \cite[Subsection I.2.3]{Ch}. We stress that this argument allows both stationary points in the original parametrization and self-intersections.

\begin{proof}
If $L_\gamma=0$, there is nothing to prove. Otherwise, we take a Lipschitz continuous arclength parametrization $\eta:[0,L_\gamma]\to\C$. After translating $\gamma$ if necessary, which does not change $\mathcal A_\gamma$ and $L_\gamma$, we may assume that $\int_0^{L_\gamma} \eta(t)\dd t=0$. Cauchy--Schwarz and the periodic Wirtinger inequality give
\[
 |\mathcal A_\gamma|
 \leq\frac12\|\eta\|_{L^2(0,L_\gamma)}\|\eta'\|_{L^2(0,L_\gamma)}
 \leq\frac{L_\gamma}{4\pi}\|\eta'\|_{L^2(0,L_\gamma)}^2
 =\frac{L_\gamma^2}{4\pi}.
\]
This gives the claimed inequality.
\end{proof}

In what follows, we write $\T:=\partial\D$ and let $\mu$ be normalized arclength measure on $\T$.

\begin{lemma}\label{lem:geometric}
Suppose that $f$ is analytic in a neighborhood of $\overline\D$, with $f(0)=0$ and $D(f)=1$. Then
\begin{equation}\label{eq:geometric}
 \int_\T e^{|f|^2}\Rea(zf'\overline f) \dd\mu \leq
 \left(\int_\T e^{|f|^2/2} |f'||f| \dd\mu \right)^2.
\end{equation}
\end{lemma}

\begin{proof}
We set $w(\theta):=f(e^{i\theta})$ and
\begin{equation}\label{eq:phase}
 \phi(\theta) :=\int_0^\theta
       \Ima\bigl(\overline{w(\theta')}w'(\theta')\bigr)\dd \theta' \,.
\end{equation}
Since $w'(\theta)=ie^{i\theta}f'(e^{i\theta})$, Parseval's identity gives
\[
 \phi(2\pi) =2\pi\sum_{n\geq1}n|a_n|^2=2\pi\, D(f) = 2\pi \,.
\]
This, together with $\phi(0)=0$, shows that
\begin{equation}\label{eq:curve}
 \gamma(\theta):=\exp\!\left(\frac{|w(\theta)|^2}2+i\phi(\theta)\right)
 \qquad\text{for all}\ 0\leq\theta\leq2\pi \,,
\end{equation}
defines a closed smooth curve in the complex plane. Moreover,
\begin{equation}\label{eq:curve-derivative}
 \gamma'=\gamma\left(\frac{(|w|^2)'}2+i\phi'\right)
          =\gamma\,\overline w\,w' \,,
\end{equation}
because $(|w|^2)'/2=\Rea(\overline w w')$ and $\phi'=\Ima(\overline w w')$. The length and signed area of $\gamma$ are therefore
\begin{align}
 L_\gamma &=\int_0^{2\pi} |\gamma'| \dd\theta =
 \int_0^{2\pi}e^{|w|^2/2}|w||w'|\dd\theta 
 = 2\pi \int_\T e^{|f|^2/2}|f||f'|\dd\mu \,,
 \label{eq:length}\\
 \mathcal A_\gamma
 &=\frac12\int_0^{2\pi}\Ima(\overline\gamma\gamma')\dd\theta
 =\frac12\int_0^{2\pi}e^{|w|^2} \Ima(\overline w w')\dd\theta 
 = \pi \int_\T e^{|f|^2}\Rea(zf'\overline f) \dd\mu \,.\label{eq:area}
\end{align}
Applying Lemma~\ref{lem:iso} gives
\[
\int_\T e^{|f|^2}\Rea(zf'\overline f) \dd\mu
= \frac1\pi \mathcal A_\gamma
\leq \frac1\pi |\mathcal A_\gamma|\leq\frac{L_\gamma^2}{4\pi^2}
 =\left(\int_\T e^{|f|^2/2}|f||f'|\dd\mu\right)^2,
\]
which is the claimed inequality.
\end{proof}


\section{Normalized radial monotonicity}\label{sec:radial}

We recall the definition of $J(f)$ from \eqref{eq:defj}.

\begin{proposition}\label{prop:monotonicity}
Let $f\not\equiv0$ be analytic in $\D$ with $f(0)=0$. Then
\[
r \mapsto J \Bigl( \frac{f(r\,\cdot)}{\sqrt{D(f(r\,\cdot))}} \Bigr)
\]
is nonincreasing on the interval $(0,1)$ and
\[
\lim_{r\downarrow0} J \Bigl( \frac{f(r\,\cdot)}{\sqrt{D(f(r\,\cdot))}} \Bigr) = e^{1/m} \,,
\]
where $m$ is the order of the zero of $f$ at the origin. In particular,
\begin{equation}\label{eq:radial-bound}
 J\Bigl( \frac{f(r\,\cdot)}{\sqrt{D(f(r\,\cdot))}} \Bigr)\leq e^{1/m}\qquad\text{for all}\ 0<r<1 \,.
\end{equation}
\end{proposition}

\begin{proof}
Writing $f(z)=\sum_{n\geq m}a_nz^n$ with $a_m\neq0$, we have
\begin{equation}
    \label{eq:dfr}
    D(f(r\,\cdot)) =\sum_{n\geq m}n|a_n|^2r^{2n} \,.
\end{equation}
This is finite and positive for every $0<r<1$. For each $0<r<1$, the function $g_r$ defined by
\begin{equation}\label{eq:normalization}
 g_r(z):=\frac{f(rz)}{\sqrt{D(f(r\,\cdot))}}
\end{equation}
is analytic in a neighborhood of $\overline\D$, satisfies $g_r(0)=0$, and has $D(g_r)=1$. 

We compute, using Parseval's identity,
\[
A_r:=\int_\T |g_r'(e^{i\theta})|^2 \,\frac{\dd\theta}{2\pi} = \frac{\sum_{n\geq m}n^2|a_n|^2r^{2n}}{D(f(r\,\cdot))} \,.
\]
This, together with \eqref{eq:dfr}, gives
\[
r\frac{\mathrm d}{\mathrm dr}D(f(r\,\cdot)) = 2 \, D(f(r\,\cdot))\, A_r \,.
\]
As a consequence, differentiating \eqref{eq:normalization}, we obtain
\begin{equation}\label{eq:radial-equation}
 r\partial_rg_r=zg_r'-A_rg_r \,,
\end{equation}
where $g_r'$ denotes differentiation with respect to $z$. To compute the $r$-derivative of $J(g_r)$ we differentiate under the integral sign, which is justified by analyticity. We obtain
\begin{align}
 r\frac{\mathrm d}{\mathrm dr}J(g_r)
 &=2\int_\T e^{|g_r|^2}
   \Rea\bigl(\overline{g_r}(zg_r'-A_rg_r)\bigr) \dd\mu \,. \label{eq:monotonicity}
\end{align}
It follows from Lemma~\ref{lem:geometric}, applied to $g_r$, together with the Cauchy--Schwarz inequality, that
\begin{align}
    \label{eq:monotonicitycs}
    \int_\T e^{|g_r|^2}\Rea(zg_r'\overline{g_r}) \dd\mu \leq
    \left(\int_\T e^{|g_r|^2/2} |g_r'||g_r| \dd\mu \right)^2 \leq A_r \int_\T e^{|g_r|^2} |g_r|^2 \dd\mu \,.
\end{align}
Thus the right side of \eqref{eq:monotonicity} is nonpositive, proving the claimed monotonicity.

Finally, uniformly on $\overline\D$,
\[
 g_r(z)\longrightarrow\frac{a_m}{|a_m|}\frac{z^m}{\sqrt m}
 \qquad\text{as }r\downarrow0.
\]
It follows that $\lim_{r\downarrow0}J(g_r)=e^{1/m}$.
\end{proof}

The following result is needed when classifying the equality cases in Theorem \ref{thm:main}.

\begin{proposition}\label{prop:monotonicitysharp}
    In the setting of Proposition \ref{prop:monotonicity}, if equality holds in \eqref{eq:radial-bound} for some $r = r_0 \in(0,1)$, then $f(z) = cz^m$ for some $c\in\C\setminus\{0\}$.    
\end{proposition}

\begin{proof}
    We consider again $g_r$, defined by \eqref{eq:normalization}. By Proposition \ref{prop:monotonicity}, the assumed equality implies that
\[
 J(g_r)=e^{1/m}
 \qquad\text{for all}\ 0<r\leq r_0 \,.
\]
    Vanishing of the left side of \eqref{eq:monotonicity} at any $r\in(0,r_0)$ implies equality in the Cauchy--Schwarz inequality in \eqref{eq:monotonicitycs} and therefore gives a constant $C_r>0$ such that
\begin{equation}\label{eq:CS-equality}
 |g_r'|^2=C_r|g_r|^2e^{|g_r|^2}\qquad\text{on }\T.
\end{equation}
The equality holds pointwise because both sides are continuous.

Note that $zf'(z)/f(z)$ tends to $m\neq 0$ as $|z|\to 0$. (Indeed, this follows by writing $f(z)=z^mu(z)$, where $u$ is analytic and $u(0)=a_m\neq0$, and noting that $zf'(z)/f(z) = m + zu'(z)/u(z)$. Consequently, there is a $\rho>0$ such that the function
\[
    T(z):=\frac{zf'(z)}{f(z)}
\]
is analytic and nonzero for $|z|<\rho$. Choose the analytic logarithm with value zero at the origin and write
\begin{equation}\label{eq:log-T}
 \log\frac{T(z)}m=\sum_{k\geq1}\ell_kz^k.
\end{equation}
For $0<r<\min\{r_0,\rho\}$, division of \eqref{eq:CS-equality} by $|g_r|^2$ and use of \eqref{eq:normalization} yield
\begin{equation}\label{eq:log-equality}
 \log|T(re^{i\theta})|^2
 =\log C_r+\frac{|f(re^{i\theta})|^2}{D(f(r\,\cdot))} \,.
\end{equation}
There are no zeros on these sufficiently small circles, so all divisions are legitimate.

Writing
\[
\log|T(re^{i\theta})|^2
= 2\log m+
2\operatorname{Re}\log\frac{T(re^{i\theta})}{m}
=2\log m+\sum_{k\ge1}
\left(\ell_kr^ke^{ik\theta}
+\overline{\ell_k}r^ke^{-ik\theta}\right),
\]
we see that the coefficient of $e^{ik\theta}$ on the left side of \eqref{eq:log-equality} is $\ell_kr^k$. On the right side, we write again $f(z) = \sum_{n\geq m} a_n z^n$ with $a_m\neq 0$ and find that the coefficient of $e^{ik\theta}$ for $k\geq1$ is
\[
 \frac1{D(f(r\,\cdot))}\sum_{n\geq m}a_{n+k}\overline{a_n}\,r^{2n+k} \,.
\]
Thus,
\[
\ell_k \, D(f(r\,\cdot)) = \sum_{n\geq m}a_{n+k}\overline{a_n}\,r^{2n} \,.
\]
Inserting the power series expansion for $D(f(r\,\cdot))$ from \eqref{eq:dfr} and comparing coefficients, we find that
\begin{equation}\label{eq:coefficient-relation}
 a_{n+k}\overline{a_n}
 =n\ell_k|a_n|^2
 \qquad\text{for all}\ n\geq m\ \text{and}\ k\geq1 \,.
\end{equation}
Equivalently, whenever $a_n\neq0$,
\begin{equation}\label{eq:recurrence}
 a_{n+k}=n\ell_ka_n.
\end{equation}

We argue by contradiction and assume that $a_{m+j}\neq0$ for some $j\geq1$. Taking $n=m$ and $k=j$ in \eqref{eq:recurrence} shows that $\ell_j\neq0$. Applying the same relation successively gives
\[
 a_{m+j}=m\ell_ja_m,
 \qquad
 a_{m+2j}=m(m+j)\ell_j^2a_m\neq0 \,.
\]
On the other hand, the relation with $n=m$ and $k=2j$ gives
\[
 \ell_{2j}=(m+j)\ell_j^2 \,.
\]
We can now compute $a_{m+3j}$ in two ways: with $(n,k)=(m+j,2j)$ and with $(n,k)=(m+2j,j)$. These give, respectively,
\[
 a_{m+3j}=m(m+j)^2\ell_j^3a_m,
 \qquad
 a_{m+3j}=m(m+2j)(m+j)\ell_j^3a_m.
\]
Since $j\geq1$ and $\ell_ja_m\neq0$, this is a contradiction. Hence $a_n=0$ for every $n>m$, and $f$ is a monomial of degree $m$, as claimed.
\end{proof}


\section{Proof of Theorem \ref{thm:main}}\label{sec:bound}

We now prove the inequality in Theorem~\ref{thm:main}. First suppose that $D(f)=1$. Then, by \eqref{eq:dfr}, $f$ belongs to the Hardy space and, in particular, it has radial limits almost everywhere on $\T$. Moreover, the same equation implies that $D(f(r\,\cdot)\to 1$ as $r\uparrow1$ and therefore $g_r$, defined again by \eqref{eq:normalization}, satisfies $g_r\to f$ almost everywhere on $\T$. By Fatou's lemma and Proposition~\ref{prop:monotonicity},
\begin{equation}\label{eq:fatou}
 J(f)\leq\liminf_{r\uparrow1}J(g_r)\leq e^{1/m} \,.
\end{equation}
This proves the claimed bound for $D(f)=1$.

For $0<D(f)<1$, apply \eqref{eq:fatou} to $\tilde f:=f/\sqrt{D(f)}$. H\"older's inequality gives
\begin{equation}\label{eq:jensen}
 J(f)=\int_\T\bigl(e^{|\tilde f|^2}\bigr)^{D(f)} \dd\mu
 \leq \left(\int_\T e^{|\tilde f|^2} \dd\mu \right)^{D(f)}
 \leq e^{D(f)/m}.
\end{equation}
This proves \eqref{eq:main}. If $f(z)=cz^m$ with $c\in\C$, then $D(f)=m|c|^2$ and $J(f)=e^{|c|^2}$, so equality is attained, and the assumption $D(f)\leq 1$ is satisfied if and only if $|c|\leq1/\sqrt{m}$. It remains to show that there are no other equality cases.

Suppose first that $D(f)=1$ and $J(f)=e^{1/m}$. By Proposition \ref{prop:monotonicity}, the function $r\mapsto J(g_r)$ is nonincreasing and its limit at zero is $e^{1/m}$. Therefore, \eqref{eq:fatou} implies that it is constant on $(0,1)$. Therefore, by Proposition \ref{prop:monotonicitysharp}, $f(z) = cz^m$ for some $c\in\C\setminus\{0\}$. The assumption $D(f)=1$ implies that $|c|=1/\sqrt m$.

Finally, if $0<D(f)<1$ and equality holds in \eqref{eq:main}, both inequalities in \eqref{eq:jensen} have to be equalities. In particular, $J(f/\sqrt{D(f)})=e^{1/m}$, so the result just proved shows that $f(z)/\sqrt{D(f)}=c z^m$ for some $c\in\C\setminus\{0\}$. This completes the proof of Theorem~\ref{thm:main}.


\section*{Acknowledgments}
R.L.F.~acknowledges partial support through the German Research Foundation grants EXC-2111-390814868 and TRR 352--Project-ID 470903074. P.I.~acknowledges partial support from the US NSF CAREER grant DMS-2152401, US NSF grant DMS-2554183, a Simons Fellowship, and a Humboldt Research Fellowship for Experienced Researchers. The authors acknowledge the use of AI tools. All mathematical arguments and proofs in the final manuscript were checked and written by the authors.


\end{document}